\documentclass[12pt]{amsart}
\newtheorem{innercustomthm}{Theorem}
\newenvironment{customthm}[1]
  {\renewcommand\theinnercustomthm{#1}\innercustomthm}
  {\endinnercustomthm}
\usepackage{amsmath,amsfonts,amsthm,amssymb,nicefrac}
\usepackage{graphicx,color,url}
\usepackage{accents}
\newtheorem{theorem}{Theorem}
\newtheorem{lemma}[theorem]{Lemma}
\newtheorem{corollary}[theorem]{Corollary}
\newtheorem{proposition}[theorem]{Proposition}
\theoremstyle{definition}

\numberwithin{equation}{section}\numberwithin{theorem}{section}
\allowdisplaybreaks
\newcounter{stepctr}
{\end{list}}

\def\XXint#1#2#3{{\setbox0=\hbox{$#1{#2#3}{\int}$}
     \vcenter{\hbox{$#2#3$}}\kern-.5\wd0}}

\newcommand{\mc}[1]{\mathcal{#1}}
\newcommand{\mbb}[1]{\mathbb{#1}}

\newcommand{\circo}{\accentset{\circ}}
\providecommand{\abs}[1]{\lvert#1\rvert}

\newcommand{\M}{\mathcal{M}}

\newcommand{\R}{\mbb{R}}

\newcommand{\Rp}{\text{Rm}^{\perp}}
\newcommand{\ra}{\rangle}
\newcommand{\la}{\langle}
\newcommand{\e}{\varepsilon}

\newcommand{\p}{\partial}

\begin{document}
\title[Cylindrical Estimates for High Codimension MCF]{Cylindrical Estimates for High Codimension Mean Curvature Flow}
\author{Huy The Nguyen}
\address{School of Mathematical Sciences\\
Queen Mary University of London\\
Mile End Road\\
London E1 4NS}
\email{h.nguyen@qmul.ac.uk}
\subjclass[2000]{Primary 53C44}

\begin{abstract}
We study high codimension mean curvature flow of a submanifold $\M^n$ of dimension $n$ in Euclidean space $\R^{n+k}$ subject to the quadratic curvature condition $ |A|^{2}\leq c_n |H|^{2}, c _n = \min\{ \frac{4}{3n} , \frac{1}{n-2}\}$. This condition extends the notion of two-convexity for hypersurfaces to high codimension submanifolds. We analyse singularity formation in the mean curvature flow of high codimension by directly proving a pointwise gradient estimate. We then show that near a singularity the surface is quantitatively cylindrical. 
\end{abstract}
\maketitle
\section{Introduction}
Let $ F_0: \mc M^n \rightarrow \R^{n+k}$ be a smooth immersion of a compact manifold $\mc M^n $. The mean curvature flow starting from $\mc M_0$ is the following family of submanifolds
\begin{align*}
F:\mc M \times [0, T) \rightarrow \R^{n+k}
\end{align*}  
such that
\begin{align*}
\frac{\partial}{\partial t}F (p,t) & = H(p,t), \quad p \in \mc M, t \in [0,T)\\
F(p,0) &=F_0(p)
\end{align*}
where $ H(p,t)$ denotes the mean curvature vector of $ \mc M_t= F_t(p)=F(p,t)$ at $p$. It is well known that this forms a system of (weakly) parabolic partial differential equations. Geometrically, the mean curvature flow is the steepest descent flow for the area functional of a submanifold and hence is a natural curvature flow.
   
Flow by mean curvature has been investigated by many authors since Huisken's seminal paper, \cite{Hu84} (see also Brakke \cite{Brakke1978}). A key central theme of mean curvature flow (and geometric flows in general) is the formation of singularities. In the case of mean curvature flow of mean convex hypersurfaces in Euclidean space  White \cite{White2000}, \cite{White2003} and Huisken-Sinestrari \cite{HuSi99a}, \cite{HuSi99b}, \cite{HuSi09} have developed a deep and far reaching analysis of the formation of singularities. The purpose of this paper is to obtain a suitable generalisation of these results for high codimension submanifolds.

A crucial step in the study of singularity formation in the mean convex mean curvature flow is the convexity estimate. This states that in regions of large mean curvature, the second fundamental form is almost positive definite. A similar estimate exists for three dimensional Ricci flow (Hamilton-Ivey estimate) \cite{Hamilton1995a} and for four manifolds with positive isotropic curvature \cite{Hamilton1997} (for recent work on PIC see \cite{Brendle2016}, \cite{Brendle2017}). The estimate in the mean curvature flow case following Huisken-Sinestrari uses Stampacchia iteration, the Michael-Simons-Sobolev inequality together with recursion formulae for symmetric polynomials. In contrast, White uses compactness theorems of geometric measure theory together with the rigidity of strong maximum principle for the second fundamental form. Haslhofer-Kleiner \cite{Haslhofer2017},\cite{Haslhofer2017a} develops an alternative approach to White's results based on Andrews's non-collapsed result for the mean curvature flow. We note that the Huisken-Sinestrari approach does not work for positive mean curvature in dimension two. The analysis of singularities and resulting mean curvature flow with surgery was developed independently by \cite{BrHu16} and \cite{Haslhofer2017a}.

Most of the work done on mean curvature flow in higher codimension uses assumptions on the image of the Gauss map. They have either considered graphical submanifolds, \cite{Chen2002},\cite{Li2003}, \cite{Wang2002},\cite{Wang2004}, submanifolds with additional symplectic or Lagrangian structure \cite{Smoczyk2002}, \cite{Chen2001},\cite{Wang2001},\cite{Smoczyk2004}, \cite{Neves2007} or using that convex subsets of the Grassmannian are preserved  by the mean curvature flow, \cite{Tsui2004},\cite{Wang2003},\cite{Wang2005}.

We will consider conditions on the second fundamental form. In high codimension, the mean curvature flow is more complex than the hypersurface case. Firstly, the presence of the normal curvature greatly complicates reaction terms in the evolution equations for the second fundamental form. Secondly, there do not seem to be many useful invariant conditions.

In this paper, we will extend the singularity analysis of \cite{HuSi09},\cite{HuSi99a} to higher codimension. However, in high codimension, since we have no direct equivalent of positive mean curvature, we need to consider an alternative condition. Note that on a compact submanifold if $ H>0$, there there exists a $c>0$ such that $ |A|^2 \leq c |H|^2$ and in fact this bound is preserved by the (co-dimension one) mean curvature flow. In fact such a condition makes sense for all codimensions. This lead Andrews-Baker \cite{Andrews2010} to consider the following pinching condition on the second fundamental form in high codimension
\begin{align*}
|A|^2 - c|H|^2<0, \ c < \frac{4}{3n}, \ |H| >0
\end{align*}
which was shown to be preserved by the mean curvature flow. For $c = \min\{\frac{4}{3n},\frac{1}{n-1}\}$, remarkably they were able to prove convergence to a round sphere. We note that the condition $ |A| -\frac{1}{n-1}H^2 <0, H>0$ implies convexity in codimension one. In this paper, we will wish to study singularity formation in high codimension mean curvature flow and will consider the following curvature pinching   
\begin{equation}
\label{eqn_pinching}
|A|^2-c_n |H|^2 \leq-\varepsilon |H|^2
\end{equation}
for some $\varepsilon >0$, where 
\begin{equation*}
c_n := \min\left \{\dfrac{4}{3n}, \frac{1}{n-2}\right \},\quad \text{if $n\geq 5$}.
\end{equation*}
In particular, note that 
\begin{align*}
c_n = \left\{
\begin{array}{cc}
\nicefrac{4}{15}, & n = 5, \\
\nicefrac{4}{18}, & n = 6, \\
\nicefrac{4}{21}, & n = 7, \\
\nicefrac{1}{n-2}, & n \geq 8.
\end{array}
\right.
\end{align*}
Such a solution is said to be quadratically bounded (we note here that this condition implies that $\M$ satisfies positive isotropic curvature) . It is very important to note that $ \frac{3}{n+2}-c_n > 0$ (for all $n\geq 2$). Our method is a mixture of B. White \cite{White2003},\cite{White2000} (as well as Haslhofer-Kleiner \cite{Haslhofer2017}) and Huisken-Sinestrari \cite{HuSi99b},\cite{HuSi09}. In Huisken-Sinestrari the mean convexity condition is used to derive asymptotic convexity. This in turn is used to derive the asymptotic cylindrical estimates. And finally, the key gradient estimate follows from asymptotic cylindricality. In high codimension, we do not have a condition on the second fundamental form like convexity. However, we do have a notion of cylindricality and while it turns out that we can prove asymptotic cylindrical estimates without first proving asymptotic convexity, we instead directly prove the pointwise gradient estimates using the quadratic curvature condition.

\begin{customthm}{\ref{thm_gradient}}\label{thm:mainThm}
Let $ \M_t , t \in [0,T)$ be a closed $n$-dimensional quadratically bounded solution to the mean curvature flow in the Euclidean space, $ \R^{n+k}$ with $n \geq 2$, that is 
\begin{align*}
|A|^2 - c|H|^2 <0, |H| >0  
\end{align*}
with $ c \leq \frac{4}{3n}$.
 Then there exists a constant $ \gamma_1= \gamma_1 (n, \mc M_0)$ and a constant $ \gamma_2 = \gamma_2 ( n , \M_0)$ such that the flow satisfies the uniform estimate 
\begin{align*}
|\nabla A|^2 \leq \gamma_1 |A|^4+\gamma_2
\end{align*}
for all $ t\in [0, T)$.
\end{customthm}

Once we have shown the gradient estimate, we can use the Arzela-Ascoli theorem to derive local compactness theorems. We then can derive asymptotic cylindrical estimates following White \cite{White2003} by using the rigidity of the evolution equation for the second fundamental form.

\begin{customthm}{\ref{thm_cylindrical}} [cf \cite{HuSi09}]
Let $ F:[0,T)\times \mc M^n \rightarrow \R^{n+k}$ be a smooth solution to the mean curvature flow such that $ F_0(p) = F(0,p)$ is compact and quadratically bounded. Then for all $ \e >0$ there is a $H_0>0$ such that if $ |H(p,t)| \geq H_0$ then 
\begin{align*}
\frac{|A|^{2}}{|H|^2}\leq \left (\frac{1}{n-1}+\e  \right)
\end{align*} 
\end{customthm}
The above constants are optimal for $n \geq 8$ as the following example shows. Consider the submanifold
\begin{align*}
\mbb S^{n-2}(\e) \times \mbb S^2(1) \subset \R^{n-1}\times \R^{3}
\end{align*}
with $ \e$ small and positive. The second fundamental form of this submanifold is given by 
\begin{align*}
h\mid_{(\e x, y)}& = 
\left ( 
\begin{array}{ccccc}
\frac{1}{\e}& & &&\\
& \ddots & & & \\
& & \frac1 \e  & & \\
 & & &0 & \\
 & & & & 0\\ 
\end{array}
\right) ( x,0 )
+\left( 
\begin{array}{ccccc}
0 & & & & \\
& \ddots & & & \\
& &  0 & &  \\
& & & 1 & \\
& & & & 1
\end{array}
\right) (0,y)
\end{align*}
We compute that 
\begin{align*}
\frac{|A|^2}{|H|^2}= \frac{n-1+\e^2}{( n-1+\e )^2}\geq \frac{1}{n-2}. 
\end{align*}
We see that if $\frac{|A|^2}{|H|^2}> \frac{1}{n-2}$ then the mean curvature flow does not necessarily develop $\mbb S^{n-1}\times \R$ cylindrical singularities as the above examples are asymptotically $ \mbb S^{n-2}\times \R^2$. Also we can repeat the same construction above with $S^{n-1}(\e) \times \mbb S^1(1)\subset \R^n \times \R ^ 2 $ satisfying $ \frac{|A|^2}{|H|^2} = \frac{ n + \e ^ 2}{ (n+\e)^2)} \geq \frac{1}{n-1}$ which lies in the interior of our curvature cone. Furthermore $\mbb S ^{n-1} \times \R\subset \R ^{n+1}$ with the standard metric also lies in the interior of our curvature cone. Hence using standard connected sum arguments, connected sums of $ \mbb S ^{n-1} \times \mbb S^1 \#\cdots \#  \mbb S ^{n-1} \times \mbb S^1$ lie in the interior on the above curvature condition. Hence nontrivial topologies are allowed by the quadratic curvature condition. In fact, in \cite{Nguyen2018}, by developing a mean curvature flow with surgery in high codimension, we will show that the above topologies, together with $\mbb S^n$ are the only allowed topologies. Finally for $ 2\leq n \leq 7$, it is likely that the above above coefficients are not optimal. In particular, for $n=2$ as shown in \cite{Baker2017}, direct control on the normal curvature in the curvature condition is likely to be required to obtain the optimal pinching condition.       

Furthermore, we can show that not only that a singularity of a high codimension mean curvature flow that the quotient $ \frac{|A|^2}{|H|^2}$ approaches that of the cylinder, that in fact unless the entire blowup surface becomes compact, there are parts of the submanifold that become arbitrarily close to the standard flat cylinder $ \mbb S^{n-1}\times \R \subset \R^{n+1}\subset \R^{n+k}$. Remarkably, the singularity that forms is  codimension one.    


This paper is set out as below. In section two we gather up the relevant preliminary results and fix our notation. In section three, we prove the crucial gradient estimates. The key to the following estimate is that we have $\frac{3}{n+2}-\frac{1}{n-2}> 0. $ The gradient estimate the allows us to control the mean curvature and hence the full second fundamental form on a neighbourhood of fixed size. In section four, we then prove the cylindrical estimates. This uses the rigidity of the evolution equation for $ \frac{|A|^2}{|H|^2}$ reminiscent of B. White's proof of the asymptotical convexity. We also show that unless the entire submanifold becomes spherical, there exists regions of the submanifold that up to the first singular time become arbitrarily close in $C^\infty$ to the standard codimension one cylinder.   

\section{Notation and preliminary results}
We adhere to the notation of \cite{Andrews2010}. A fundamental ingredient in the derivation of the evolution equations is Simons' identity:
\begin{equation}\label{eqn:SimonsId}
\Delta h_{ij}=\nabla_i\nabla_jH+H\cdot h_{ip}h_{pj}-h_{ij}\cdot h_{pq}h_{pq}+2h_{jq}\cdot h_{ip}h_{pq}
-h_{iq}\cdot h_{qp}h_{pj}-h_{jq}\cdot h_{qp}h_{pi}.
\end{equation}
The timelike Codazzi equation combined with Simons' identity produces the evolution equation for the second fundamental form:
\begin{equation}\label{eqn:evolA}
\nabla_{\partial_t}h_{ij}= \Delta h_{ij}+h_{ij}\cdot h_{pq}h_{pq}+h_{iq}\cdot h_{qp}h_{pj}
+h_{jq}\cdot h_{qp}h_{pi}-2h_{ip}\cdot h_{jq}h_{pq}.
\end{equation}
The evolution equation for the mean curvature vector is found by taking the trace with $g_{ij}$:
\begin{equation}
\nabla_{\partial_t}H=\Delta H+H\cdot h_{pq}h_{pq}.
\end{equation}
The evolution equations of the norm squared of the second fundamental form and the mean curvature vector are
\begin{align}
	\frac{\p}{\p t}\abs{A}^2 &= \Delta\abs{A}^2-2\abs{\nabla A}^2+2 \sum_{\alpha, \beta}\bigg( \sum_{i,j}h_{ij\alpha}h_{ij\beta}\bigg)^2+2 \sum_{i,j,\alpha,\beta}\bigg( \sum_p h_{ip\alpha}h_{jp\beta}-h_{jp\alpha}h_{ip\beta}\bigg)^2 \label{eqn:A2}\\
	\frac{\partial}{\partial t}\abs{H}^2 &= \Delta\abs{H}^2-2\abs{\nabla^\bot H}^2+2\sum_{i,j}\bigg( \sum_{\alpha}H_{\alpha}h_{ij\alpha}\bigg)^2. \label{eqn:H2}
\end{align}
The last term in \eqref{eqn:A2} is the squared length of the normal curvature, which we denote by $\abs{\Rp}^2$. For convenience we label the reaction terms of the above evolution equations by
\begin{gather*}
	R_1 = \sum_{\alpha, \beta}\bigg(\!\sum_{i,j}h_{ij\alpha}h_{ij\beta}\!\bigg)^2+\abs{\Rp}^2 \\
	R_2 = \sum_{i,j}\!\bigg(\!\sum_{\alpha}H_{\alpha}h_{ij\alpha}\bigg)^2.
\end{gather*}

\subsection{Preservation of pinching}
We consider the quadratic quantity
\begin{align}
\mc Q = |A|^2+a-c |H|^2 
\end{align}
where $ c$ and $ a$ are positive constants. Combining the evolution equations for $ |A|^2 $ and  $| H|^2$ yields
\begin{align}
\label{eqn_pinchpres}
\partial_t \mc Q & = \Delta \mc Q-2 ( |\nabla A |^2-c |\nabla H|^2 )+2 R_1-2 c R_2.
\end{align}
We have the following Kato type inequality which is a consequence of the Codazzi equation.
 \begin{lemma}
 For any hypersurface $\M_0 \subset \mathbb{R}^{n+k}$ we have that 
 \begin{align}\label{ineqn_Kato}
|\nabla A|^{2}\geq \frac{3}{n+2}| \nabla H|^2.
\end{align}
 \end{lemma}
This is proven in \cite{Andrews2010} (as in Hamilton \cite{Hamilton1982} and Huisken \cite{Hu84}) and shows
that the gradient terms in \eqref{eqn_pinchpres} are strictly negative if $ c < \frac{3}{n+2}$. For $c < \frac{4}{3n}$ we also have $R_1-cR_2 < 0$ (see \cite{Andrews2010}), so by the maximum principle:

\begin{lemma}
Let $F:\mc M^n \times[0,T)\rightarrow \R^{n+k}$ be a solution to the mean curvature flow such that $\mc M_0$ satisfies 
\begin{align*}
|A|^2+a \leq  c |H|^2 
\end{align*}
for some $a > 0$ and $ c \leq \frac{4}{3 n}$. Then this condition is preserved by the mean curvature flow.  
\end{lemma}

In fact we will require a slight modification of this Lemma. For the above reaction terms in the evolution of \eqref{eqn_pinchpres}, since $\mc Q \leq 0$ is preserved by the flow, we have 
\begin{align*}
2 R_1-2 c R_2  &\leq2 |A_1|^2 \mc Q-2 a |A_1|^2-\frac{2a}{n}\frac{1}{c-\nicefrac{1}{n}}|\circo A_-|^2 \\ 
&+\frac{2}{n}\frac{1}{c-\nicefrac{1}{n}}| A_-|^2 \mc Q+\left(6-\frac{2}{n (c-\nicefrac{1}{n})}  \right) |\circo A_1|^2 | \circo A_-|^2+\left(3-\frac{2}{n (c-\nicefrac{1}{n})}  \right)|\circo A_-|^4
\end{align*}
where $A_1$ is the second fundamental form in the mean curvature direction and $ \circo A_-$ represents (traceless) second fundamental form in the directions orthogonal to the mean curvature. Therefore we have the following Lemma, 
\begin{lemma}
Let $F:\mc M^n \times[0,T)\rightarrow \R^{n+k}$ be a solution to the mean curvature flow such that $\mc M_0$ satisfies 
\begin{align*}
\mc Q(x,0)=|A|^2(x,0)+a- c |H|^2(x,0) 
\end{align*}
for some $a > 0$ and $ c \leq \frac{4}{3 n}$. Then $\mc Q(x,t) \leq 0 $ and we have the following evolution inequality 
\begin{align}
\label{eqn_pinchpres2}
\nonumber\partial_t \mc Q & \leq  \Delta \mc Q-2 ( |\nabla A |^2-c |\nabla H|^2 )+2 |A_1|^2 \mc Q-2 a |A_1|^2-\frac{2a}{n} \frac{1}{c-\nicefrac{1}{n}}|\circo A_-|^2 \\ 
 &+\frac{2}{n}\frac{1}{c-\nicefrac{1}{n}}| A_-|^2 \mc Q+ \left(6-\frac{2}{n (c-\nicefrac{1}{n})}  \right) |\circo A_1|^2 | \circo A_-|^2+\left(3-\frac{2}{n (c-\nicefrac{1}{n})}  \right)|\circo A_-|^4\\
\nonumber &\leq 0.
\end{align}

\end{lemma}
In particular, the reaction terms satisfy $R_1-c R_2\leq  0 $ whenever $ \mc Q\leq 0$.     
As a consequence, we see that the flow preserves both $|H| >0$ and \eqref{eqn_pinching}. 

The following existence theorem holds for the mean curvature flow of $\mc M_0$ under the conditions of Theorem \ref{thm:mainThm}:
\begin{theorem}\label{thm:longTimeExistence}
The mean curvature flow of $\M_0$ exists on a finite maximal time interval $0 \leq t < T < \infty$.  Moreover, $\limsup_{t\rightarrow T}\abs{A}^2 = \infty$.
\end{theorem}
The proof that the maximal time of existence is finite follows easily from the evolution equation for the position vector $F$: $\frac{\partial}{\partial t}\abs{F}^2 = \Delta\abs{F}^2-2n$.  The maximum principle implies $\abs{F(p,t)}^2\leq R^2-2nt$ and thus $T\leq \frac{R^2}{2n}$, where $R=\max\left\{\abs{F_0(p)}:\ p\in\Sigma\right\}$. The proof of the second part of the theorem can be found in \cite{Andrews2010}.




\section{Gradient Estimates}
In this section, we prove a gradient estimate for the mean curvature flow. We prove this estimate directly from the quadratic curvature bound, $ |A| < c |H|^2, c \leq \frac{4}{3n}$ without needing the asymptotic cylindrical estimates. (In fact we will prove the cylindrical estimates as a consequence of the gradient estimates). The estimates derived here are pointwise gradient estimate, they depend only on the mean curvature (or equivalently the second fundamental form) at a point and not on the maximum of curvature as for more general parabolic type derivative estimates of \cite{Ecker1991},\cite{Andrews2010}. We have   
\begin{align*}
\frac{3}{n+2}-c> 0. 
\end{align*}
The inequality allows us to combine the derivative terms in the evolution equation of $ |A|^2-c |H|^2$ with the Kato-type inequality $ |\nabla A|^2 \geq \frac{3}{n+2}|\nabla H|^2 $. 
\begin{theorem}\label{thm_gradient}
Let $ \M_t , t \in [0,T)$ be a closed $n$-dimensional quadratically bounded solution to the mean curvature flow in the Euclidean space, $ \R^{n+k}$ with $n \geq 2$, that is 
\begin{align*}
|A|^2 - c|H|^2 <0, |H| >0  
\end{align*}
with $ c \leq \frac{4}{3n}$.
 Then there exists a constant $ \gamma_1= \gamma_1 (n, \mc M_0)$ and a constant $ \gamma_2 = \gamma_2 ( n , \M_0)$ such that the flow satisfies the uniform estimate 
\begin{align*}
|\nabla A|^2 \leq \gamma_1 |A|^4+\gamma_2
\end{align*}
for all $ t\in [0, T)$.
\end{theorem}

\begin{proof}
We choose here $ \kappa_n = \left( \frac{3}{n+2}-c\right)>0$. As $\frac{1}{n}\leq c\leq \frac{4}{3n}$, $n\geq 2,\kappa_n$ is strictly positive. We will consider here the evolution equation for 
\begin{align*}
\frac{|\nabla A|^2}{g^2}
\end{align*}
where $ g = \alpha |H|^2-|A|^2-\beta $ where $ \alpha $ and $\beta$ are constants coming from the quadratic bounds on $|A|^2$ ensuring that $g$ is strictly positive. Since $ |A|^2-c|H|^2 < 0, |H|>0$ and $\M_0$ is compact, there exists an $ \eta(\M_0) >0, C_\eta(\M_0)>0$ so that 
\begin{align}\label{eqn_eta}
\left( c-\eta \right)|H|^2-|A|^2 \geq C_\eta>0.
\end{align} 
 Hence we set
 \begin{align*}
g = c|H|^2-|A|^2 >\e |A|^2>0
\end{align*}
where $ \e = \frac{\eta}{c}$. From the evolution equations,  \eqref{eqn_pinchpres2}, we get 
\begin{align*}
\partial_t g &= \Delta g-2 \left(  c|\nabla H|^2-|\nabla A|^2 \right)+2  \left( c   R_2-R_1 \right)\\
 &\geq \Delta g-2  \left(  \frac{n+2}{3}c-1 \right) |\nabla A|^2 \\
 & \geq \Delta g+2\kappa_n  \frac{n+2}{3}| \nabla A|^2.  
\end{align*}
The evolution equation for $ |\nabla A|^2 $ is given by 
\begin{align*}
 \frac \partial{\partial t}|\nabla A|^2-\Delta |\nabla A|^2 &\leq-2 |\nabla^2 A|^2+c_n |A|^2 |\nabla A|^2.
\end{align*}
Let $w,z$ satisfy the evolution equations
\begin{align*}
\frac{\partial}{\partial t}w = \Delta w+W , \quad \frac{\partial}{\partial t}z = \Delta z+Z 
\end{align*}
 then we find that 
 \begin{align*}
\partial_t \left(\frac{w}{z}\right) &= \Delta\left( \frac{w}{z}\right) +\frac{2}{z}\left \la \nabla \left( \frac{w}{z}\right) , \nabla z \right \ra+\frac{W}{z}-\frac{w}{z^2} Z\\
&= \Delta\left( \frac{w}{z}\right) +2\frac{\la \nabla w , \nabla z \ra}{z^2}-2 \frac{w|\nabla z |^2}{z^3}+\frac{W}{z}-\frac{w}{z^2} Z.
\end{align*}
Furthermore for any function $g$, we have by Kato's inequality
\begin{align*}
\la \nabla g , \nabla |\nabla A|^2 \ra &\leq 2 |\nabla g| |\nabla^2 A| |\nabla A| \\
&\leq \frac{1}{g}|\nabla g |^2 | \nabla A|^2+g |\nabla^2 A|^2.     
\end{align*}
We then get
\begin{align*}
-\frac{2}{g}| \nabla^2 A|^2+\frac{2}{g}\left \la \nabla g ,\nabla \left( \frac{|\nabla A|^2}{g}\right) \right\ra \leq-\frac{2}{g}| \nabla^2 A|^2-\frac{2}{g^3}|\nabla g|^2 |\nabla A|^2+\frac{2}{g^2}\la \nabla g ,\nabla |\nabla A|^2 \ra \leq 0 .  
\end{align*}
Then if we let $ w = |\nabla A|^2 $ and $ z = g$ with $W \leq-2 |\nabla^2 A|^2+c_n |A|^2 |\nabla A|^2$ and $Z\geq 2\kappa_n  \frac{n+2}{3}| \nabla A|^2 $ we get 
\begin{align*}
\frac{\partial}{\partial t}\left( \frac{|\nabla A|^2}{g}\right)-\Delta \left( \frac{|\nabla A|^2}{g}\right) &\leq \frac{2}{g}\left \la \nabla g ,\nabla \left( \frac{|\nabla A|^2}{g}\right) \right \ra+\frac{1}{g}(-2 |\nabla^2 A|^2+c_n |A|^2 |\nabla A|^2 ) \\
&-2 \kappa_n \frac{n+2}{3}\frac{|\nabla A|^4}{g^2} \\
& \leq c_n |A|^2  \frac{|\nabla A|^2}{g}-2 \kappa_n \frac{n+2}{3}\frac{|\nabla A|^4}{g^2}.
\end{align*}
We repeat the above computation with $w = \frac{|\nabla A|^2}{g}, z = g,$
\begin{align*}
W\leq c_n |A|^2  \frac{|\nabla A|^2}{g}-2 \kappa_n \frac{n+2}{3}\frac{|\nabla A|^4}{g^2}\end{align*} 
and $ Z\geq 0$ to get 
\begin{align*}
\frac{\partial}{\partial t}\left( \frac{|\nabla A|^2}{g^2}\right)-\Delta \left( \frac{|\nabla A|^2}{g^2}\right) &\leq \frac{2}{g}\left \la \nabla g ,\nabla \left( \frac{|\nabla A|^2}{g^2}\right) \right \ra  \\
&+\frac{1}{g}\left (   c_n |A|^2 \frac{|\nabla A|^2}{g}-2 \kappa_n \frac{n+2}{3}\frac{|\nabla A|^4}{g^2}\right)
\end{align*}
The nonlinearity then is
\begin{align*}
\frac{|\nabla A|^2}{g^2} \left( c_n-\frac{2 \kappa_n(n+2)}{3}\frac{|\nabla A|^2}{|A|^2 g}  \right)\leq\frac{|\nabla A|^2}{g^2} \left( c_n-\frac{2 \kappa_n(n+2) \e}{3}\frac{|\nabla A|^2}{g^2}  \right).
\end{align*}
By the maximum principle, there exists a constant (with $\eta$ chosen sufficiently small so that this estimate holds at the initial time) such that 
\begin{align*}
\frac{|\nabla A|^2}{g^2}\leq \frac{3 c_n}{2 \kappa_n (n+2) \e}.  
\end{align*}
Therefore we see that there exists a constant $C = \frac{ 3c_n }{2 \kappa_n (n+2)\e  }= C(n, \M_0) $ such that
\begin{align*}
\frac{|\nabla A|^2}{g^2}\leq C.
\end{align*} 
\end{proof}

We will also want to control the time derivative of curvature with constants with explicit dependence. In order to do so, we now derive quantitative estimates for the second derivative of curvature.   
\begin{theorem}
Let $ \mc M$ be a solution of the mean curvature flow then there exists constants $\gamma_3, \gamma_4$  depending only on the dimension and pinching constant so that 
\begin{align*}
|\nabla^2 A|^2 \leq \gamma_3 |A|^6+\gamma_4. 
\end{align*} 
\end{theorem}
\begin{proof}
We have the following evolution equation
\begin{align*}
\partial_t |\nabla^2 A|^2-\triangle |\nabla^2 A|^2 &\leq-2 |\nabla^3 A|^2+k_1 |A|^2 |\nabla^2 A|^2+k_2 |A| |\nabla A|^2 | \nabla^2 A|\\
&\leq-2 |\nabla^3 A|^2+\left( k_1+\frac{k_{2}}{2}\right) |A|^2 | \nabla^2 A|^2+\frac{k_2}{2}|\nabla A|^4.
\end{align*}
We now consider the evolution equation of the term $ \frac{| \nabla^2  A |^2}{|H|^5}$. 
Firstly we see that 
\begin{align*}
\partial_t |H|^\alpha \geq \triangle |H|^\alpha-\alpha(\alpha-1) |H|^{\alpha-2}|\nabla H|^2. 
\end{align*}
Therefore we get 
\begin{align*}
\partial_t \frac{|\nabla^2 A|^2}{|H|^5}-\triangle \frac{| \nabla^2 A|^2}{|H|^5}&\leq \frac{1}{|H|^5}\left (-2 |\nabla^3 A|^2+\left (k_1+\frac{k_2}{2}\right )|A|^2 |\nabla^2 A|^2+\frac{k_2}{2}|\nabla A|^4   \right)\\
&+\frac{20|H|^3 |\nabla^2 A|^2 |\nabla H|^2} {|H|^{10}}+\frac{2}{|H|^{10}}\la \nabla |H|^5 , \nabla |\nabla^2 A|^2 \ra-\frac{2 |\nabla^2 A|^2}{|H|^{15}}| \nabla |H|^5 |^2.  
\end{align*}
We have the terms 
\begin{align*}
 \frac{20|H|^3 |\nabla^2 A|^2 |\nabla H|^2} {|H|^{10}}-\frac{2 |\nabla^2 A|^2}{|H|^{15}}| \nabla |H|^5 |^2  &\leq \frac{20 |\nabla^2 A |^2 |\nabla H|^2}{|H|^7}-\frac{50 |\nabla^2  A|^2 |\nabla |H| |^2}{| H|^7}\\
& \leq   \frac{20 |\nabla^2 A |^2 |\nabla H|^2}{|H|^7}.
\end{align*}
and 
\begin{align*}
\frac{2}{|H|^{10}}\la \nabla |H|^5 , \nabla |\nabla^2 A|^2 \ra &= \frac{10 \la \nabla |H|, \nabla |\nabla^2 A|^2 \ra }{|H|^6}\\
& \leq \frac{1}{|H|^5} |\nabla^3 A|^2 +\frac{100|\nabla H|^2 | \nabla^2 A|^2}{|H|^7}   
\end{align*}
Together with the gradient estimate, this gives the following evolution equation
\begin{align*}
\partial_t \frac{|\nabla^2 A|^2}{|H|^5}-\triangle \frac{| \nabla^2 A|^2}{|H|^5}&\leq-\frac{|\nabla^3 A|^2}{|H|^5}+k_3 \frac{|\nabla^2 A|^2}{|H|^4}+\frac{120 |\nabla H|^2 |\nabla^2 A|^2}{|H|^7}+\frac{k_2}{2}|\nabla A|^4   \\
&\leq-\frac{ \nabla^3 A|^2}{|H|^5}+k_4 \frac{| \nabla^2 A|^2}{|H|^3}+C_1\frac{|\nabla^2 A|^2}{|H|^7}+\frac{k_5 |H|^8+C_2}{|H|^5}   
\end{align*}
Similar computations give us 
\begin{align*}
\partial_t \frac{|\nabla A|^2}{|H|^3}-\triangle \frac{|\nabla A|^2}{|H|^3}&\leq-\frac{|\nabla^2 A|^2}{|H|^3}+\frac{k_6 |H|^8+C_3}{|H|^5}, \\
 \partial_t \frac{|\nabla A|^2}{|H|^7}-\triangle \frac{|\nabla A|^2}{|H|^7}& \leq-\frac{|\nabla^2 A|^2}{|H|^7}+\frac{k_7 |H|^8+C_4}{| H|^9}.  
\end{align*}
We now set 
\begin{align*}
f = \frac{|\nabla^2 A|^2}{|H|^5}+N \frac{|\nabla A|^2}{|H|^3}+M \frac{|\nabla A|^2}{|H|^7}-\kappa \sqrt{ c|H|-|A|^2} 
\end{align*}
We have 
\begin{align*}
\partial_t f-\triangle f &\leq k_4 \frac{|\nabla^2 A|^2}{|H|^3}+k_5 |H|^3+C_1 \frac{|\nabla^2 A|^2}{|H|^7}+\frac{C_2}{|H|^5}\\
&-N \frac{|\nabla^2 A|^2}{|H|^3}+N k_6 |H|^3+\frac{N C_3}{|H|^5} \\
&-\frac{M|\nabla^2 A|^2}{|H|^7}+\frac{k_7 M}{|H|}+\frac{C_4 M}{|H|}-\kappa \e_0|H|^3.  
\end{align*}
Therefore we choose 
\begin{align*}
N>k_4, \quad \e_0 \kappa>  Nk_6+k_5 , \quad M > C_1, 
\end{align*}
Then we find, recalling that $ |H| \geq \alpha_1$ that
\begin{align*}
\partial_t f-\triangle f \leq C_5 
\end{align*}
which implies that
 \begin{align*}
\max_{\mc M_t }f \leq \max_{\mc M_{t_0}}f+C_5 (t-t_0)
\end{align*}
which gives the desired estimate.
\end{proof}

As a special case of our estimates we get the following statement. This corollary will be used extensively in the analysis of regions of high curvature. Note that a similar estimate plays an important role in the Ricci flow with surgery, \cite[Equation (1.3)]{Perelman2003}. That estimate is obtained via a compactness argument and is qualitative in nature. 
\begin{corollary}
Let $ \mc M_t$ be a mean curvature flow. Then there exists $c^\#, H^\#>0$ such that for all $p \in \mc M $ and $t>0$ which satisfy
\begin{align*}
|H(p,t)| \geq H^\# \implies |\nabla H(p,t)| \leq c^\# |H(p,t)|^2, \quad |\partial_t H(p,t)|\leq c^\# |H(p,t)|^3.
\end{align*}
\end{corollary}

Note that the following Lemma is purely a statement concerning submanifolds subject to gradient estimate for the mean curvature and is not concerned with mean curvature flow.
\begin{lemma}\label{lem_localHarnack}
Let $ F: \mc M^n \rightarrow \R^{n+k}$ be an immersed submanifold. Suppose there exists $ c^\#, H^\# $ such that
\begin{align*}
|\nabla H(p)|^2 \leq c^{\#}|H(p)|^2
\end{align*}
for any $ p\in \mc M $ such that $ |H|(p) \geq H^\#$. Let $ p_0\in \mc M$ satisfy $ |H(p_0)| \geq \gamma H^{\#}$ for some $ \gamma >1 $. Then
\begin{align*}
|H(q)| \geq \frac{|H(p_0)|}{1+c^{\#}d ( p_0, q)}\geq \frac{|H(p_0)|}{\gamma},
\end{align*}
$ \quad \forall q \mid d(p_0,q) \leq \frac{\gamma-1}{c^{\#}}\frac{1}{|H|(p_0) }$.
\end{lemma}


\begin{proof}
We firstly consider the case where we have points $ \bar q \in \mc M$ such that $ | H(\bar q)| < \frac{|H(p_0)|}{\gamma}$. Let $q_0$ be a point with minimal distance from $ p_0$. Furthermore let $ d_0 = d ( p_0, q_0)|H(p_0)|$ and $\theta_0 = \min \left\{d_0, \frac{\gamma-1}{c^\#} \right\}$ (later we will show that $ d_0 \geq  \frac{\gamma-1}{c^\#}$). 

Let $ q \in \mc M$ be any point with $d ( q, p_0) \leq \frac{\theta_0}{| H(p_0)|}$. Then consider a minimising geodesic $ \zeta: [0, d(p_0, q)]\rightarrow \mc M$ parameterised by arclength. By the definition of $ \theta_0$ we see that 
\begin{align*}
d(p, \zeta(s)) < d(p,q) \leq \frac{\theta_0}{| H(p_0)|}\leq \frac{d_0}{| H(p_0)|}= d (p_0,q_0)
\end{align*}
which implies that 
\begin{align*}
|H(\zeta(s) )| \geq \frac{| H(p_0) |}{\gamma}\geq H^{\#}\quad \forall s \in [ 0, d ( p_0, q) ].
\end{align*}
We apply the gradient estimate to get 
\begin{align*}
|\nabla H(\zeta(s) ) | \leq c^{\#}|H ( \zeta(s) )|^2 .
\end{align*}
Differentiating this expression we get 
\begin{align*}
\frac{d}{ds}| H( \zeta(s)) | \geq-c^{\#}|H ( \zeta(s) )|^2 \quad   \forall s \in [ 0, d ( p_0, q) ].
\end{align*}
Integrating this differential inequality, we get 
\begin{align*}
|H(\zeta(s) )| \geq \frac{| H(p_0) |}{1+c^{\#}s |H(p_0) |}\quad  \forall s \in [ 0, d ( p_0, q) ].
\end{align*}
This implies that 
\begin{align*}
|H(q)| \geq \frac{|H(p_0)|}{1+c^{\#}d ( p_0, q) |H(p_0)|}
\end{align*}
as 
\begin{align*}
1+c^{\#}d (p_0,q)|H(p_0)| \leq 1+c^{\#}d_0 < 1+c^{\#}\theta_0
\end{align*}
which holds if $ d(p_0, q ) \leq \frac{\theta_0}{|H(p_0)|}$. Now suppose that $ d_0 < \frac{\gamma-1}{c^{\#}}$ so that $\theta_0 = d_0$. But then 
\begin{align*}
d_0 = d(p_0, q_0) | H(p_0)| 
\end{align*}
so that 
\begin{align*}
d(p_0,q_0) = \frac{\theta_0}{|H(p_0)|}.
\end{align*}
But we have 
\begin{align*}
d_0< \frac{\gamma-1}{c^{\#}}\implies 1+c^{\#}d_0 < \gamma \implies \frac{1}{\gamma}< \frac{1}{1+c^{\#}d_0} 
\end{align*}
which implies that 
\begin{align*}
|H(q_0)| \geq \frac{|H ( p_0 )|}{1+c^{\#}d_0}> \frac{| H(p_0)|}{\gamma}
\end{align*}
which is a contradiction. Hence $ d_0 \geq \frac{\gamma-1}{c^{\#}}$ and so $\theta_0 = \frac{\gamma-1}{c^{\#}}$. 

If instead we have that $| H(\bar q ) | \geq \frac{| H(p_0) |}{\gamma} \ \forall \bar q \in \mc M$ then
\begin{align*}
|\nabla H| \leq c^{\#}|H|^2
\end{align*}
everywhere on $ \mc M$. Hence the same argument above applies as well. 
\end{proof}

We  denote by $ g(t)$ the metric induced on $ \mc M \subset \R^{n+k}$ at time $t$. Given $p\in \mc M$ and $ r >0$, we let $ B_{g(t)}(p,r)\subset \mc M$ the closed ball of radius $r$ about $p$ with respect to the metric $g(t)$. In addition if $ t, \theta$ are given such that $0\leq t-\theta < t \leq T_0$, we set
\begin{align*}
\mc P (p,t,r,\theta) = \left \{(q,s) \mid q \in B_{g(s)}(p,r), s \in [t-\theta,t] \right\}.
\end{align*}
Such a set is called a (backward) parabolic neighbourhood of $ (p,t)$. Given a point $ (p,t)$, we set
\begin{align*}
\hat{r}(p,t) \equiv \frac{n-1}{|H (p,t)|}, \quad \hat{\mc P}(p,t,L,\theta) \equiv  \mc P ( p, t ,\hat r(p,t)L, \hat r(p,t)^2 \theta ).    
\end{align*}
It is well known that estimates such as above give local Harnack estimates, that is local control on the size of the curvature in a neighbourhood of a given point. 
\begin{lemma}
Let $ c^{\#}, H^{\#}$ be constants of Lemma \ref{lem_localHarnack}. We define $ d^{\#}= \frac{1}{8(n-1)^2 c^{\#}}$. If $ (p,t)$ satisfies $ |H|(p,t) \geq H^{\#}$ then given any $r, \theta \in (0, d^\#]$ then
\begin{align*}
\frac{|H(p,t)|}{2}\leq |H(q,s)| \leq 2 |H(p,t)| 
\end{align*}
for all $ (q,s) \in \hat P (p,t,r,\theta)$. 
\end{lemma}

\section{Cylindrical Estimates}
In the analysis of singularities of mean curvature flow, a crucial step are convexity estimates. For the mean convex mean curvature flow of hypersurfaces, it says that points of large curvature have almost positive definite fundamental form. For the three dimensional Ricci flow, there is a similar estimate known as the Hamilton-Ivey pinching estimate. Huisken-Sinestrari prove this estimate by a complex iteration scheme using the symmetric polynomials of second fundamental form with $L^p$ estimates, the Michael-Simon Sobolev inequality. For high codimension, we do not have convexity. Rather, we prove that a singularity which is not spherical must be cylindrical. We prove this by using the gradient estimate. This is in contrast to \cite{HuSi09}, where the cylindrical estimates are derived as consequence of the convexity estimates of \cite{HuSi99b}. In this section, we will be interested in cylindrical singularities of form $ \mbb{S}^{n-1} \times \mbb{S}^1 \subset \R^{n+k}$. Therefore we will consider pinching of the form
\begin{align*}
|A|^2-c _n |H|^2 < 0, n \geq 5  
\end{align*}
  where $ c _n = \min\left \{\frac{4}{3n}, \frac{1}{n-2} \right\}$. 
\begin{theorem}[\cite{HuSi09}]\label{thm_cylindrical}
Let $ F:[0,T)\times \mc M^n \rightarrow \R^{n+k}$ be a smooth solution to the mean curvature flow such that $ F_0(p) = F(0,p)$ is compact and quadratically bounded. Then for all $ \e >0$ there is a $H_0>0$ such that if $ |H(p,t)| \geq H_0$ then 
\begin{align*}
\frac{|A|^{2}}{|H|^2}\leq \left (\frac{1}{n-1}+\e  \right)
\end{align*} 
\end{theorem}

\begin{proof}
Let $ T_{\max}$ denote the first singular time. By parabolic regularity, we must have that 
\begin{align*}
\lim_{t\rightarrow T_{\max}}\sup_{\mc M_t}|A(p,t) |^2 =+\infty.
\end{align*}
Furthermore, since 
\begin{align*}
\frac{1}{n}|H|^2< |A|^2 < \frac{1}{n-2}| H|^2 
\end{align*}
the second fundamental form $A$ and the mean curvature $H$ have the same blow up rate, so we must have 
\begin{align*}
\lim_{t\rightarrow T_{\max}}\sup_{\mc M_t}|H(p,t) |^2 =+\infty.
\end{align*}
Suppose the estimate is not true. Then there exists an $\e>0$ where we have 
\begin{align*}
\limsup_{t\rightarrow T_{\max}}\sup_{p \in \mc M_t}\frac{|A(p,t) |^2}{|H(p,t) |^2}= \frac{1}{n-1}+ \e > \frac{1}{n-1}.
\end{align*}
This limit exists since we have 
\begin{align*}
\frac{|A|^2}{|H|^2}\leq \frac{1}{n-2}.
\end{align*}
Furthermore there exists a sequence of points $ p_k $ and times $ t_k $ such that as $ k \rightarrow \infty $, $ t_k \rightarrow T_{\max}$  
and 
\begin{align*}
\lim_{k\rightarrow \infty}\frac{|A(p_k, t_k ) |^2}{|H(p_k, t_k ) |^2}= \frac{1}{n-1}+ \e .
\end{align*}
We perform a parabolic rescaling of $ \bar M_t^k $ in such a way that the norm of the mean curvature at $(p_k,t_k)$ becomes $n-1$. That is, if $F_k$ is the parameterisation of the original flow $ \mc M_t^k $, we let $ \hat r_k = \frac{n-1}{|H(p_k,t_k)|}$, and we denote the rescaled flow by $ \mc M_t^k $ and we define it as 
\begin{align*}
\bar F_k (p,\tau) = \frac{1}{\hat r_k}( F_k ( p , \hat r_k^2 \tau+t_k)-F_k (p_k , t_k) )
\end{align*}  

For simplicity, we choose for every flow a local co-ordinate system centred at $ p_k$. In these co-ordinates we can write $0$ instead of $ p_k$. The parabolic neighbourhoods $\mc P^k ( p_k, t_k, \hat r_k L, \hat r_k^2 \theta)$ in the original flow becomes $ \bar{\mc P}(0,0,\theta, L)$. By construction, each rescaled flow satisfies 
\begin{align*}
\bar F_k (0,0) = 0, \quad |\bar H_k (0,0) | = n-1. 
\end{align*}
The gradient estimates give us uniform bounds on $ |A|$ and its derivatives up to any order on a neighbourhood of the form $\bar{\mc P}( 0 ,0,d,d)$ for a suitable $ d > 0$. This gives us uniform estimates in $ C^\infty $ on $ \bar F_k$. Hence we can apply Arzela-Ascoli and conclude that there exists a subsequence that converges in $ C^\infty $ to some limit flow which we denote by $ \tilde M_\tau^\infty$. We now analyse the limit flow $ \tilde M_\tau^\infty$. Note that we have 
\begin{align*}
\bar A_k ( p , \tau ) = \hat r_k A_k ( p , \hat r_k^2 \tau+t_k ). 
\end{align*}
so that
\begin{align*}
\frac{|\bar A_k(p, \tau) |^2}{|H_k(p,\tau ) |^2}& = \frac{|A_k ( p, \hat r_k^2 \tau+t_k ) |^2}{|H_k(p, \hat r_k^2 \tau+t_k ) |^2} 
\end{align*}
but since $ \hat r_k \rightarrow 0, t_k \rightarrow T_{\max}$ as $ k\rightarrow \infty $ this implies that 
\begin{align*}
\frac{|\bar A ( p , \tau ) |^2}{| \bar  H (p, \tau ) |^2}=\lim_{k\rightarrow \infty}\frac{|\bar A_k(p,\tau)|^2}{|\bar H_k ( p, \tau ) |^2} \leq \frac{1}{n-1} +\e
\end{align*}
and 
\begin{align*}
\frac{| \bar A(0,0)|^2}{| \bar H(0,0)|^2}= \frac{1}{n-1}+\e.
\end{align*}
Hence the flow $\bar{\mc M}_t^\infty $ has a space-time maximum for $\frac{|\bar A ( p , \tau ) |^2}{| \bar  H (p, \tau ) |^2}$ at $ (0,0)$. Since the evolution equation for $ \frac{|A|^2}{|H|^2}$ is given by 
\begin{align*}
\partial_t \left ( \frac{|A|^2}{|H|^2}\right)-\triangle \left (\frac{|A|^2}{|H|^2}\right) & = \frac{2}{|H|^2}\left \la \nabla |H|^2 , \nabla \left ( \frac{|A|^2}{|H|^2}\right) \right \ra\\
&-\frac{2}{|H|^2} \left ( |\nabla A|^2-\frac{|A|^2}{|H|^2}|\nabla H|^2 \right) \\
&+\frac{2}{|H|^2}\left( R_1-\frac{|A|^2}{|H|^2} R_2\right)
\end{align*}
Now we have that 
\begin{align*}
|\nabla H|^2 \leq \frac{3}{n+2}|\nabla A|^2, \quad  \frac{|A|^2}{|H|^2}\leq c_n
\end{align*}
which gives
\begin{align*}
-\frac{2}{|H|^2} \left ( |\nabla A|^2-\frac{|A|^2}{|H|^2}|\nabla H|^2 \right) \leq 0.
\end{align*}
Furthermore if $\frac{|A|^2}{|H|^2}=c <  c_n$ then 
\begin{align*}
 R_1-\frac{|A|^2}{|H|^2} R_2=  R_1-c  R_2 \leq 0.
\end{align*}
Hence the strong maximum principle applies to the evolution equation of $\frac{|A|^2}{|H|^2}$ and shows that $\frac{|A|^2}{|H|^2}$ is constant. The evolution equation then shows that $ |\nabla A|^2= 0$, that is the second fundamental form is parallel. Finally this shows that locally $ \mc M = \mbb S^{n-k}\times \R^k$, \cite{Lawson1969}. As $\frac{|A|^2}{|H|^2}< c_n\leq  \frac{1}{n-2}$ we can only have 
\begin{align*}
\mbb S^n, \mbb S^{n-1}\times \R
\end{align*}
which gives $\frac{|A|^2}{|H|^2}= \frac{1}{n}, \frac{1}{n-1}$ which is a contradiction. 
\end{proof}


\begin{lemma}
Let $ \mc M_t , t \in [0, T)$ be a mean curvature flow. Let $\e, \theta, L, k \geq k_0$ be given where $k_0\geq 2$. Then we can find $ \eta_0>0, H_0>0$ with the following properties. Suppose that 
\begin{align*}
|H( p_0, t_0)|\geq H_0, \quad \&\quad \frac{|A|^2}{|H|^2}( p_0, t_0 ) \geq \frac{1}{n-1}-\eta_0   
\end{align*}   
then the neighbourhood 
\begin{align*}
\hat{\mc P}( p_0, t_0, L-1, \nicefrac{\theta}{2})
\end{align*}
is an $(\e,k,L-1, \frac{\theta}{2})$ shrinking neck.
\end{lemma}
\begin{proof}
We proceed by a proof by contradiction. Hence there exists $ \e , L, \theta$ such that the conclusion of the theorem is not true. This implies that we can find a sequence of mean curvature flows $\{\mc M_t^j \}_{j \geq 1}$ and sequence of times $ \{t_j \}_{j \geq 1}$ and points $ \{p_j\}_{j \geq 1}, p_j \in \mc M_{t_j}^j$ such that for the space-time points $ \{p_j, t_j\}_{j\geq 1}$ at the image points $ F_j( p_j, t_j) \in \mc M_{t_j}^j$ then as $ t \rightarrow T_{\max}, |H(p_j, t_j )| \rightarrow \infty  $ and
\begin{align*}
\frac{|A|^2(p_j ,t_j )}{|H|^2 ( p_j ,t_j )}\rightarrow \frac{1}{n-1}
\end{align*}  
but $ ( p_j, t_j )$ does not lie at the centre of a $ ( \e_k, L-1, \frac{\theta}{2})$ neck. Note that $ \frac{|A|^2}{|H|^2}<c _n \leq \frac{1}{n-2}$ so the gradient estimates apply. 

We perform a parabolic rescaling of $ \bar M_t^k $ in such a way that the norm of the mean curvature at $(p_k,t_k)$ becomes $n-1$. That is, if $F_k$ is the parameterisation of the original flow $ \mc M_t^k $, we let $ \hat r_k = \frac{n-1}{|H(p_k,t_k)|}$, and we denote the rescaled flow by $ \mc M_t^k $ and we define its parameterisation as 
\begin{align*}
\bar F_k (p,\tau) = \frac{1}{\hat r_k}( F_k ( p , \hat r_k^2 \tau+t_k)-F_k (p_k , t_k) )
\end{align*}  

For simplicity, we choose for every flow a local co-ordinate system centered at $ p_k$. In these co-ordinates we can write $0$ instead of $ p_k$. The parabolic neighbourhoods $\mc P^k ( p_k, t_k, \hat r_k L, \hat r_k^2 \theta)$ in the original flow becomes $ \bar{\mc P}(0,0,L,\theta)$. By construction, each rescaled flow satisfies 
\begin{align*}
\bar F_k (0,0) = 0, \quad |\bar H_k (0,0) | = n-1. 
\end{align*}
The gradient estimates give us uniform bounds on $ |A|$ and its derivatives up to any order on a neighbourhood of the form $\bar{\mc P}( 0 ,0,d,d)$ for a suitable $ d > 0$. This gives us uniform estimates in $ C^\infty $ on $ \bar F_k$. Hence we can apply Arzela-Ascoli and conclude that there exists a subsequence that converges in $ C^\infty $ to some limit flow which we denote by $ \tilde M_\tau^\infty$. We now analyse the limit flow $ \tilde M_\tau^\infty$. Note that we have 
\begin{align*}
\bar A_k ( p , \tau ) = \hat r_k A_k ( p , \hat r_k^2 \tau+t_k ),
\end{align*}
so that
\begin{align*}
\frac{|\bar A_k(p, \tau) |^2}{|H_k(p,\tau ) |^2}& = \frac{|A_k ( p, \hat r_k^2 \tau+t_k ) |^2}{|H_k(p, \hat r_k^2 \tau+t_k ) |^2} 
\end{align*}
but since $ \hat r_k \rightarrow 0, t_k \rightarrow T_{\max}$ as $ k\rightarrow \infty $ this implies that 
\begin{align*}
\frac{|\bar A ( p , \tau ) |^2}{| \bar  H (p, \tau ) |^2}=\lim_{k\rightarrow \infty}\frac{|\bar A_k(p,\tau)|^2}{|\bar H_k ( p, \tau ) |^2} \leq \frac{1}{n-1}
\end{align*}
and 
\begin{align*}
\frac{| \bar A(0,0)|^2}{| \bar H(0,0)|^2}= \frac{1}{n-1 }.
\end{align*}
Hence the flow $\bar{\mc M}_t^\infty $ has a space-time maximum for $\frac{|\bar A ( p , \tau ) |^2}{| \bar  H (p, \tau ) |^2}$ at $ (0,0)$. Since the evolution equation for $ \frac{|A|^2}{|H|^2}$ is given by 
\begin{align*}
\partial_t \left ( \frac{|A|^2}{|H|^2}\right)-\triangle \left (\frac{|A|^2}{|H|^2}\right) & = \frac{2}{|H|^2}\left \la \nabla |H|^2 , \nabla \left ( \frac{|A|^2}{|H|^2}\right) \right \ra\\
&-\frac{2}{|H|^2} \left ( |\nabla A|^2-\frac{|A|^2}{|H|^2}|\nabla H|^2 \right) \\
&+\frac{2}{|H|^2}\left( R_1-\frac{|A|^2}{|H|^2} R_2\right).
\end{align*}
Now we have that 
\begin{align*}
|\nabla H|^2 \leq \frac{3}{n+2}|\nabla A|^2, \quad  \frac{|A|^2}{|H|^2}<c _n\leq \frac{1}{n-2}
\end{align*}
which gives
\begin{align*}
-\frac{2}{|H|^2} \left ( |\nabla A|^2-\frac{|A|^2}{|H|^2}|\nabla H|^2 \right) \leq 0.
\end{align*}
Furthermore if $\frac{|A|^2}{|H|^2}=c<c _n  \leq  \frac{1}{n-2}$ then 
\begin{align*}
 R_1-\frac{|A|^2}{|H|^2} R_2=  R_1-c  R_2 \leq 0.
\end{align*}
Hence the strong maximum principle applies to the evolution equation of $\frac{|A|^2}{|H|^2}$ and shows that $\frac{|A|^2}{|H|^2}$ is constant. The evolution equation then shows that $ |\nabla A|^2= 0$, that is the second fundamental form is parallel. Finally this shows that locally $ \mc M = \mbb S^{n-k}\times \R^k$, (see \cite{Lawson1969}). As $\frac{|A|^2}{|H|^2}< \frac{1}{n-2}$ we can only have 
\begin{align*}
\mbb S^n, \mbb S^{n-1}\times \R
\end{align*}
which gives $\frac{|A|^2}{|H|^2}(x)= \frac{1}{n}, \frac{1}{n-1}$ for every $(p,t)$ which is a contradiction. 

\end{proof}

The following theorem and proposition are independent of mean curvature flow and concern the compactness of manifolds (and submanifolds) subject to a curvature inequality. 
\begin{theorem}[Bonnet-Myers, Hopf-Rinow] 
Let $ \mc M$ be a complete Riemannian manifold and suppose that $ p \in \mc M $ such that the sectional curvature satisfies $ K > K_{\min}$ along all geodesics of length $\frac{\pi}{K_{\min}}$ from $p$ or equivalently $K > K_{\min}$ in a neighbourhood $d(p,q) \leq \frac{\pi}{\sqrt{K_{\min}}}$.    

\end{theorem}
\begin{proposition}[B.Y Chen \cite{Chen1993}] 
For $ n \geq 2 $ if $ \M^n$ is a submanifold of $ \R^{n+k}$ then at every $ p \in \M^n$, 
\begin{align*}
K_{\min}\geq \frac{1}{2}\left( \frac{1}{n-1}|H(p)|^2-|A(p)|^2\right).
\end{align*}

\end{proposition}

The following theorem shows that either the surface is compact or there exists a neck region.  
\begin{theorem}
Let $ F: \mc M \rightarrow \R^{n+k}$, be a smooth connected immersed submanifold. Suppose that there exists $c^{\#}, H^{\#}>0$ such that $ | \nabla H( p ) | \leq c^{\#}|H(p)|^2 $ for all $ p \in \mc M$ where $ | H(p) | \geq H^{\#}$. Then $\forall \eta_0 >0$ there exists $\alpha_0 =  \alpha_0(c^\#,\eta_0 ), \gamma_0 = \gamma_0 (c^\#,\eta_0 )$, if $ p \in \mc M $ satisfies 
\begin{align*}
\frac{| A|^2}{|H|^2}( p) < \frac{1}{n-1}-\eta_0, \quad |H(p)| \geq \gamma_0 H^{\#}.  
\end{align*}  

Then either $ \mc M$ is closed with 
\begin{align*}
\frac{|A|^2}{|H|^2}< \frac{1}{n-1}-\eta_0 \quad \text{everywhere}
\end{align*}
or there exists a $q \in \mc M$ such that
\begin{enumerate}
\item \begin{align*}
\frac{|A|^2}{|H|^2}(q) \geq \frac{1}{n-1}-\eta_0 
\end{align*}

\item 
\begin{align*}
d(p,q) \leq \frac{\alpha_0}{|H(p) |}
\end{align*}

\item 
\begin{align*}
|H(q')| \geq \frac{|H(p) |}{\gamma_0}, \quad \forall q'\in \mc M \mid d(p,q') \leq \frac{\alpha_0}{|H(p)|}   
\end{align*}
in particular $ | H(q) | \geq \frac{|H(p) |}{\gamma_0}$.   
\end{enumerate}
\end{theorem}
\begin{proof}
Given $ \alpha_0$,  we set $ \gamma_0 = 1+c^\#\alpha_0$. Then for a given $ p\in\mc M$, we let
\begin{align*}
\mc M_{p,\alpha_0}= \left\{q \in \mc M \mid d(p,q) \leq \frac{\alpha_0}{|H(p)|}\right\}.
\end{align*}
Then we have that if $ |H(p) | \geq \gamma_0 H^\#$, then
\begin{align*}
|H(q)| \geq \frac{|H(p) |}{1+c^\#d(p,q) |H(p)|}\geq \frac{|H(p) |}{\gamma_0}. 
\end{align*}
We show that if $ \alpha_0$ is sufficiently large then this implies $\mc M$ is compact. Suppose that we have 
\begin{align*}
\frac{|A|^2}{|H|^2}(q) < \frac{1}{n-1}-\eta_0, \quad \forall q \in \mc M_{q, \alpha_0}.
\end{align*}

We have that on $\mc M_{q, \alpha_0}$
\begin{align*}
K_{\min}\geq \frac{1}{2}\left ( \frac{1}{n-1}|H(p)|^2-|A(p)|^2 \right)\geq \frac 1 2 \eta_0 | H( p ) |^2\geq \frac{\eta_0}{2}H_{\min}^2> 0.   
\end{align*}
Applying the above lemma, if we choose $ \alpha_0 = \frac{\sqrt{2}\pi}{\sqrt{\eta_0}} $ then we have that 
\begin{align*}
K_{\min} > 0 \quad \text{in a neighbourhood } d(p,q) \leq \frac{\alpha_0}{| H( p ) |}\leq  \frac{\pi}{K_{\min}}
\end{align*}
which shows that $ \mc M $ is compact.
\end{proof}

The following corollary shows that before the first singular time either a cylindrical neck exists or the flow becomes spherical. 
\begin{corollary}
Let $ \mc M_t $ be a smooth mean curvature flow of a closed submanifold with $ |A|^2 \leq c_n|H|^2$. Given neck parameters $ \e,k,L$ there exists $ H^*$ depending on the initial data such that $ \sup_{p \in \mc M_{t_0}}H(p, t_0) \geq H^*$ then the submanifold at time $ t_0$ either contains an $ ( \e, k, L )$-cylindrical neck or is spherically pinched, that is \begin{align*}
|A|^2 - \frac{1}{n-1}|H|^2 < 0. 
\end{align*}.
\end{corollary}

\end{document}